%% file: main.tex
\documentclass[11pt]{amsart}

\input{preamble}

\title{Refined invariants for Abelian surfaces: between polynomiality and modularity}
\author{Thomas Blomme}
\address{Institut Mathématique de Toulouse, 118 route de Narbonne. F-31062 Toulouse Cedex 9}
\email{thomas.blomme@cnrs.fr}
\author{Gurvan Mével}
\address{Institut de Mathématiques de Jussieu - Paris Rive Gauche, Sorbonne Université, 4 place Jussieu, 75252 Paris Cédex 5, France}
\email{gurvan.mevel@imj-prg.fr}

\begin{document}

\begin{abstract}
    Tropical refined invariants for toric surfaces, introduced Block and Göttsche, are obtained couting tropical curves with a Laurent polynomial multiplicity. Brugallé and Jaramillo-Puentes then exhibited a polynomial behavior of the coefficients of this Laurent polynomial, seen as function on the curve degree. The authors provided explicit formula for small genus, involving quasi-modular forms.

    Inspired by the toric setting, the first-named author defined refined invariants for abelian surfaces and extended the polynomiality result. In this paper, we further study this regularity for abelian surfaces, providing explicit formulas involving quasi-modular forms. This resonates with the small genus cases of the toric setting.
\end{abstract}

\maketitle

\tableofcontents

\section{Introduction}

\subsection{Setting}

\subsubsection{Toric case}

Since Mikhalkin's correspondence theorem \cite{mikhalkin2005enumerative}, counts of tropical curves have been an efficient tool in enumerative geometry. One can endow tropical curves with integers multiplicities that enable the computation of Gromov-Witten (resp. Welschinger) invariants, which are complex (resp. real) invariants obtained counting complex (resp. real) curves of given genus and degree passing through a given configuration of points. Introduced by Block and Göttsche \cite{block2016fock,block2016refined}, tropical refined invariants are Laurent polynomials (in the variable $q$) obtained counting instead tropical curves with a Laurent polynomial multiplicity. These invariants interpolate between complex and real enumerations of curves on toric surfaces as their evaluation at $q=1$ recovers Gromov-Witten invariants (or Severi degrees), while plugging $q=-1$ gives (tropical) Welschinger invariants.
The invariance of the tropical enumeration with the Block-Göttsche multiplicities was established by Itenberg and Mikhalkin \cite{itenberg2013block}.

\medskip

Given a fixed toric surface $X$ and a non-negative integer $\delta$, it was conjectured by Di Francesco-Itzykson in case $X=\PP^2$ \cite{difrancesco-1995-quantum} and more generally by Göttsche \cite{gottsche1998conjectural}, that the number of curves in $X$ with $\delta$ nodes and passing through the appropriate number of points (i.e. the Severi degree) behaves polynomially when the linear system (the ``degree'' of the curves) varies. This was first proved by Tzeng \cite{tzeng2010proof}.

\medskip

By the adjunction formula, the genus and the number of nodes play a dual role. However, if ones fixes the genus instead, polynomiality is not preserved. With G\"ottsche's conjecture in mind, Brugallé and Jaramillo-Puentes studied the coefficient of fixed \textit{codegree} of the refined invariants, fixing the genus and varying the linear system. Surprisingly, they recovered a polynomial behavior \cite{brugalle-2022-polynomiality}. The second-named author studied in more details the genus~$0$ case and showed some universal formulas for the refined invariant \cite{mevel-2026-universal}: polynomials are constant given by a power of the function of partition numbers. The authors then generalized this result to the genus $1$ case \cite{blommemevel2025asymptotic}, where the coefficients are not constant anymore, witnessing the appearance of the first Eisenstein series, suggesting an interaction of the coefficients with quasi-modular forms.

The authors conjectured that for any genus, the tropical refined invariants have bounded degree, and the coefficients of these polynomials giving the coefficients are given by quasi-modular forms \cite[Conjecture 1.1]{blommemevel2025asymptotic}. The polynomiality results for coefficients of fixed codegree were also reformulated as asymptotic statement for a corresponding generating series called \emph{asymptotic refined invariant}, see Section \ref{subsec-laurenttopoly}. These three papers \cite{brugalle-2022-polynomiality, mevel-2026-universal, blommemevel2025asymptotic} all use a floor diagrams algorithm, introduced in \cite{brugalle2007enumeration, brugalle2008floor}, or a derived combinatorial method.


\subsubsection{Abelian case}
    
Though developed in the case of toric surfaces, the use of tropical curves has since been extended to work in other cases as well, especially the case of abelian surfaces with a correspondence theorem by Nishinou \cite{nishinou2020realization}. The first-named author recently introduced tropical refined invariants in this setting \cite{blomme2022abelian1,blomme2022abelian2}. He developed a pearl diagram algorithm, a combinatorial method similar to the floor diagram one, to handle the calculations in Abelian surfaces, and established some polynomiality results \cite{blomme2022abelian3}, see Theorem \ref{theo:polynomiality} below. He also gave explicit formulas for \emph{primitive classes}. We refer to Section \ref{sec-refinedinv} for more details on the history of refined invariants for Abelian surfaces.
    
The goal of this paper is to further study the regularity of the formulas obtained in \cite{blomme2022abelian3}, and to highlight the appearance of quasi-modular forms. We hope our results support \cite[Conjecture 1.1]{blommemevel2025asymptotic}.

\subsection{Results and organisation of the paper}

In Section \ref{sec-refinedinv} we review previous works on refined invariants for Abelian surfaces. In Section \ref{sec-polyandasympt} we recall the polynomiality results of \cite{blomme2022abelian3} and introduce the \emph{asymptotic refined invariant} $AR_g^\star(n,x)$, where $g$ is the genus, $n$ the self-intersection of the curves class we look at and $x$ a formal variable. This \emph{true} polynomial is related to the tropical refined invariant (a \emph{Laurent} polynomial) by a shift of the coefficients, see sections \ref{subsec-laurenttopoly} and \ref{subsec-asymptotics} for more precise explanations. We then state our main result which is as follows. 

\begin{named}{Theorem \ref{theorem-abelian-codeg-series}}
    For fixed genus $g$, there exists quasi-modular forms $f_1,\dots,f_{g-3}$ vanishing at $0$ such that
    $$ AR_g^\star(n,x)= \bino{n}{g-1}+\sum_{k=1}^{g-3} f_k(x)n^k \in \QQ[\![x]\!][n].$$
\end{named}

The theorem is to be understood as follows. The self-intersection of the curve class $\beta$ is $2n$, and a polynomial dependence in $n$ is a particular case of polynomial dependence in $\beta$. The codegree $i$ coefficient of the refined invariant is given by a polynomial $Q_{g,i}(n)$ for $n$ big enough. The polynomial $Q_{g,i}(n)$ giving the coefficient is the $x^i$-coefficient of $AR_g^\star(n,x)$. Furthermore, the $n^l$ coefficient of $Q_{g,i}(n)$, seen as a function of $i$, are coefficients of a quasi-modular form.

\medskip

Section \ref{sec:proofs} is devoted to the proof of the theorem. The proof being quite technical, we first proceed for small values of the genus $g$. Both in particular examples and in the general case, the proof consists in a careful combinatorial analysis of the formulas of \cite{blomme2022abelian3}. Quasi-modular forms appear in Lemma \ref{lemma-Gm-quasi-modular}.

Last, in Section \ref{sec-fixedcodeg} we compute the first coefficients of the generating series with the genus as the parameter. Although it is not obvious to us that a pattern appears, we show these calculations for someone to build upon.



\tocless{\subsection*{Acknowledgments}}
This work was initiated during a visit of GM in Neuchâtel. We thank University of Neuchâtel for excellent working environment.



\section{A summary on refined invariants for abelian surfaces} \label{sec-refinedinv}

There are two ways to apprehend refined invariants of Abelian surfaces. One through the tropical picture \cite{blomme2022abelian1,blomme2022abelian2,blomme2022abelian3}, and one through (reduced) Gromov-Witten theory \cite{bryan2018curve}. We shortly review both here, and provide explicit formulas used to prove the main result of the paper.

    \subsection{Tropical refined invariants for Abelian surfaces}

    \subsubsection{Tropical tori and tropical curves}
    
    We refer to \cite{blomme2022abelian1} or to \cite{blomme2025short} for a broader introduction to tropical curves in tropical abelian varieties.

    \begin{defi}
        A tropical torus $\TT A$ is a quotient $\RR^2/\Lambda$ where $\Lambda\simeq\ZZ^2$ is a rank $2$ lattice. We denote the inclusion by $S\colon\Lambda\hookrightarrow\RR^2$. The quotient possesses a natural integral structure given by $\ZZ^2\subset\RR^2$.
    \end{defi}
    
    \begin{defi}
    A parametrized tropical curve in $\TT A$ is a map $h\colon\Gamma\to\TT A$ where
    \begin{enumerate}
        \item $\Gamma$ is a metric graph;
        \item $h$ is affine with integral slope on the edges of $\Gamma$;
        \item $h$ satisfies the balancing condition: the sum of outgoing slope at each vertex is $0$.
    \end{enumerate}
    \end{defi}

    The \textit{genus} $g$ of a parametrized tropical curve is the first Betti number of the underlying graph. Its \textit{gcd} $\delta$ is the g.c.d. of integral length of slopes of edges.
    
    Let $h\colon\Gamma\to\TT A$ be a parametrized tropical curve. If $e$ is an edge of $\Gamma$ with a chosen orientation and $u_e$ its slope following the associate orientation, so that $u_ee$ does not depend on the chosen orientation, we can consider the following $1$-chain in $\TT A$, with $\ZZ^2$-coefficients:
    $$[\Gamma]=\sum_e u_e e \in C_1(\TT A,\ZZ^2).$$
    The balancing condition implies that its boundary is $0$ and thus this $1$-chain is actually a $1$-cycle. The homology class realized by $[\Gamma]$ is called the \textit{degree} of the tropical curve. It is denoted by $B$ and belongs to $H_1(\TT A,\ZZ^2)\simeq \Lambda\otimes\ZZ^2$. In particular, $B\in \Lambda\otimes\ZZ^2$ can be also be seen as a linear map
    $$B\colon \Lambda^*\to \ZZ^2.$$

    For a generic choice of $S\colon\Lambda\hookrightarrow\RR^2$ (i.e. a generic choice of real matrix), there is usually no tropical curve in $\TT A$. When such a choice exists and we denote by $B$ its degree, assuming $\det(B)\neq 0$, the tropical torus is called a \textit{tropical abelian variety}, and $B$ is a \textit{polarization}. According to \cite[Lemma 2.9]{blomme2025short}, we have the following.

    \begin{prop}
        Let $\TT A$ be a tropical torus given by $S\colon\Lambda\hookrightarrow\RR^2$. Then $B\colon\Lambda^*\to\ZZ^2$ is a polarization if and only if $BS^\intercal\colon(\RR^2)^*\to\RR^2$ induces a symmetric and positive definite bilinear form.
    \end{prop}
    
     The condition from the previous proposition can be seen as a constraint on the choice of $S$, which amounts to choose a real matrix, with now a symmetry condition. For instance, if $B$ is the identity matrix, $S$ shall be chosen symmetric and positive definite. We denote by $\BBB$ the set of possible polarizations, i.e. the set of integer matrices with positive determinant.


    \subsubsection{Enumerative problem and invariants}
    
    We now define the refined invariants by counting tropical curves solution to a suitable enumerative problem. Let $g\geqslant 2$ and $B\in\BBB$ be fixed, and $\TT A$ a generic tropical abelian surface with polarization $B$. It is proven in \cite{blomme2022abelian1} that given a generic set $\P$ of $g$ points inside $\TT A$, there is a finite number of tropical curves passing through $\P$, which are trivalent. 
    
    \begin{defi}[\cite{blomme2022abelian3}]
    Let $h\colon\Gamma\to\TT A$ be a trivalent tropical curve with gcd $\delta$. We define its refined multiplicity by
    $$M_\Gamma = \sum_{k|\delta}\varphi(k)k^{2g-2}\prod_V (q^{m_V/2k}-q^{-m_V/2k}),$$
    where $\varphi(n)=n\prod_{p|n}\left(1-\frac{1}{p}\right)$ is the Euler function, the product is over vertices of $\Gamma$ and $m_V$ denotes the Mikhalkin's multiplicity of the vertex (the absolute value of the determinant of two out of the three outing slopes). The multiplicity is a (symmetric) Laurent polynomial in the variable $q$.
    \end{defi}
    
    We now count solutions of the enumerative problem with their refined multiplicity, setting
    $$BG_{g,B}(\TT A,\P) = \sum_{\substack{h\colon\Gamma\to\TT A \\ h(\Gamma)\supset\P}} M_\Gamma.$$
    By \cite[Theorem 4.12,4.14]{blomme2022abelian1}, $BG_{g,B}(\TT A,\P)$ does not depend on the choice of $\P$ provided it is generic, nor on the choice of $\TT A$ provided it is also generic among tropical tori with polarization given by $B$. We get the Block-G\"ottsche refined invariant $BG_{g,B}(q)$.
    
    By deformation invariance, $BG_{g,B}(q)$ only depends on the equivalence class of $B$, meaning that we can multiply $B$ by integral invertible matrices on both sides. In other words, $BG_{g,B}(q)$ only depends on $B$ through: its \textit{divisibility} (the g.c.d. of its coefficients) and its \textit{self-intersection} $2\det(B)$.


    \subsubsection{Explicit computation}
    
    The computation is enabled by the following two results:
    \begin{itemize}
        \item the multiple cover formula \cite[Theorem 4.10(iv)]{blomme2022abelian3} which reduces the computation to primitive classes (classes with g.c.d. $1$):
        $$BG_{g,B}(q) = \sum_{k|B}k^{2g-1}BG_{g,\widetilde{B/k}}(q^k),$$
        where $\widetilde{B/k}$ is a primitive polarization such that $\det(\widetilde{B/k})=\det(B/k)$;
        \item an explicit computation for primitive classes using the polarization $B=\left(\begin{smallmatrix}
            1 & 0 \\ 0 & n
        \end{smallmatrix}\right)$, yielding
        $$BG_{g,n}(q) = g\sum_{a_1+\cdots+a_{g-1}=n}P_{a_1}\cdots P_{a_{g-1}}, \quad \text{ where }P_a(q)=\sum_{k|a}\frac{a}{k}(q^k-2+q^{-k})\in\ZZ[q^{\pm 1}].$$
    \end{itemize}

In particular, the degree of $BG_{g,n}(q)$ is $n$, and the degree of $BG_{g,B}(q)$ is $\det(B)$.


    \subsubsection{Curves in a fixed linear system}
    
    It is also possible to count genus $g$ tropical curves belonging to a fixed linear system passing a configuration of $g-2$ points, which is the purpose of \cite{blomme2022abelian2}. We obtain other refined invariants denoted by $BG_{g,B}^\star(q)$. They are however related to previous $BG_{g,B}(q)$ by the short formula
    $$BG_{g,B}^\star(q) = \frac{\det(B)}{g(g-1)} \cdot BG_{g,B}(q),$$
    and we therefore do not expand more on their precise definition, though we prefer to provide an asymptotic development for the latter. We refer to \cite{blomme2022abelian2} and \cite{blomme2022abelian3} for more details.

    \subsection{Connection to reduced Gromov-Witten invariants}
    
    Another way to define refined invariants for abelian surfaces is to consider some specific generating series of reduced Gromov-Witten invariants. It is also possible to get a close formula, and check that the definition coincides with the tropical one. The connection between refined invariants and Gromov-Witten invariants was established in the toric case by Bousseau \cite{bousseau2019tropical}.


    \subsubsection{Reduced Gromov-Witten invariants}
    
    Let $A$ be a complex abelian surface, i.e. a complex torus endowed with a realizable curve class $\beta\in H_2(A,\ZZ)$. We can consider the moduli space of stable maps $\M_{g,n}(A,\beta)$ of genus $g$ stable maps with $n$ marked points realizing the homology class $\beta$. Following \cite{bryan2018curve}, it is endowed with maps
    $$\ev\colon\M_{g,n}(A,\beta)\to A^n, \quad \ft\colon\M_{g,n}(A,\beta)\to\overline{\M}_{g,n},$$
    as well as a \textit{reduced virtual class} $[\M_{g,n}(A,\beta)]^\mathrm{red}$ of dimension $g+n$ which plays the role of the fundamental class if the moduli space was a manifold \cite{bryan2018curve}. Reduced Gromov-Witten invariants are obtained by capping the pull-back of cohomology classes by $\ev$ and $\ft$ with the reduced class: if $\alpha\in H^*(\overline{\M}_{g,n},\QQ)$ and $\gamma_i\in H^*(A,\QQ)$,
    $$\bra{\alpha;\gamma_1,\dots,\gamma_n}_{g,\beta} = \int_{[\M_{g,n}(A,\beta)]^\mathrm{red}} \ft^*\alpha \cup\prod_1^n\ev_i^*\gamma_i.$$
    The reduced class is invariant by deformation, ensuring that these numbers depend on $\beta$ only through its divisibility (as an element in the lattice $H_2(A,\ZZ)$) and its self-intersection $\beta^2$.


    \subsubsection{Generating series of invariants with a $\lambda$-class}
    
    For the case of interest, we take $n=g_0$ and each $\gamma_i\in H^4(A,\QQ)$ to be the class Poincar\'e dual to a point. For the class coming from $\overline{\M}_{g,n}$, we take $\alpha=\lambda_{g-g_0}$, where the $\lambda$-classes are the Chern classes of the Hodge bundle. Recall that the Hodge bundle is the rank $g$ vector bundle whose fiber above a stable map $f\colon C\to A$ is $H^0(C,\omega_C)$. We then consider the following generating series:
    $$BG_{g_0,\beta}^\CC(u) = \sum_{g\geqslant g_0}(-1)^{g-g_0}\bra{\lambda_{g-g_0};\pt^{g_0}}_{g,\beta}u^{2g-2}.$$
    We then do the change of variable $q=e^{iu}$, or equivalently $q-q^{-1}=2i\sin(u)$, to get a Laurent series in $q$. These generating series are already considered in \cite{bryan2018curve}.


    \subsubsection{Explicit computation}
    
    When the curve class $\beta$ is primitive, the generating series is computed by \cite[Theorem 2]{bryan2018curve}. It is also possible to compute them using the reduced decomposition formula from \cite{blomme2025multiple} to get a closed expression.

    For non-primitive classes, the computation is enabled by the multiple cover formula \cite[Theorem 4.23]{blomme2025multiple}. The connection to the multiple cover formula for tropical refined invariants through the change of variable $q=e^{iu}$ is explained in \cite[Section 1.2.3]{blomme2022abelian3}.

    Using both of the above, the explicit computation recovers the tropical refined invariants: if $\beta$ and $B$ have the same divisibility and $\beta^2=2\det(B)$, then $BG_{g_0,\beta}^\CC(u)=BG_{g_0,B}(q)$ after the change of variable $q=e^{iu}$. In particular, the Laurent series obtained at the beginning of the section is in fact a Laurent polynomial of degree $\frac{\beta^2}{2}$, which is not obvious from its definition through the change of variable.

    \begin{rem}
        In the complex case as well we may consider curves belonging to a fixed linear system. According to \cite[Corollary 2.2]{bryan1999generating}, the fixed linear system condition can be traded to the insertions of a basis of $H^3(A,\ZZ)$. Using \cite[Proposition 2]{bryan2018curve}, the computation in the case of the insertion of a $\lambda$-class and point insertions reduces to the invariants previously handled in the section. Denoting with a $\star$ the generating series for invariants in a fixed linear system, we also get the relation
        $BG_{g_0,\beta}^{\CC,\star}(u) = \frac{\beta^2/2}{g_0(g_0-1)} \cdot BG^\CC_{g_0,\beta}(u)$. We thus do not expand more on this setting.
    \end{rem}

\section{Polynomiality properties and asymptotics} \label{sec-polyandasympt}

    \subsection{Polynomiality}
    
    In \cite[Section 6.3, Theorem 6.7]{blomme2022abelian3}, we have two observations on the coefficients of the refined invariants $BG_{g,B}(q)$, recapped in the following theorem. If $P$ is a Laurent polynomial of degree $d$, we denote by $\bra{P}_i$ the \textit{codegree} $i$ coefficient, i.e. the coefficient in front of $q^{d-i}$.

    \begin{theo}\cite[Theorem 6.7]{blomme2022abelian3}
    \label{theo:polynomiality}
        For $i\geqslant 0$, for $n=\det(B)$ big enough, we have the following:
        \begin{enumerate}
            \item $\bra{BG_{g,B}(q)}_i = \langle BG_{g,\widetilde{B}}(q)\rangle_i$;
            \item $n\mapsto\bra{BG_{g,n}(q)}_i$ coincides with a polynomial function $Q_{g,i}(n)$ of degree at most $g-2$.
        \end{enumerate}
    \end{theo}

In particular, Theorem \ref{theo:polynomiality}(1) states than studying the asymptotic of coefficients, we may restrict to primitive classes.

    \begin{proof}
        \begin{enumerate}
            \item We use the multiple cover formula:
            $$BG_{g,B}(q) = \sum_{k|B}k^{2g-1}BG_{g,\widetilde{B/k}}(q^k).$$
            For $k|B$, the degree of the $k$-summand is $k\cdot\det\widetilde{B/k} = k\cdot\frac{n}{k^2}=\frac{n}{k}$. So as soon as $k\geqslant 2$,
            $$\frac{n}{k} \leq n-\frac{n}{2},$$
            and its codegree is at least $\frac{n}{2}$. Provided $n$ is big enoughn this ensures that the codegree $i$ coefficient only comes from the $k=1$ term, i.e. the term associated to the primitive class $\widetilde{B}$.
            \item The second statement comes from the explicit expression of the invariant. It also follows from the more explicit computations carried out in Section \ref{sec:proofs}.
        \end{enumerate}
    \end{proof}

    The above theorem is an abelian surface version of the polynomiality statement of usual Block-G\"ottsche invariants for toric surfaces \cite{brugalle-2022-polynomiality}.

    \subsection{From Laurent polynomials to polynomials} \label{subsec-laurenttopoly}

Adopting a point of view similar to \cite{blommemevel2025asymptotic}, we reformulate the polynomiality result from Theorem \ref{theo:polynomiality} as an asymptotic statement on the function $B\mapsto BG_{g,B}(q)$. To do so, we do a change of variable to transform the Laurent polynomial in $q$ into a polynomial in $x$ of degree $2\det(B)$. The formula consists in a mere shift of coefficients, so that the codegree $i$ coefficient now becomes the $x^i$-coefficient:
$$\overline{BG}_{g,B}(x) = x^{\det(B)}BG_{g,B}\left(\frac{1}{x}\right)\in\ZZ[x].$$
More generally, if $P\in\ZZ[q^{\pm 1}]$ is a Laurent polynomial in the variable $q$ of degree $d$, we set
$$\bar{P}(x)=x^dP\left(\frac{1}{x}\right),$$
which is now a true polynomial of degree $2d$ in the $x$ variable. The map $P(q)\in\ZZ[q^{\pm 1}]\mapsto\bar{P}(x)$ is multiplicative (and linear on Laurent polynomials of the same degree).

We now have the formula
\[ \overline{BG}_{g,n}(x) = g\sum_{a_1+\cdots+a_{g-1}=n}\bar{P}_{a_1}\cdots\bar{P}_{a_{g-1}}, \quad \text{where }\bar{P}_a(x) = \sum_{k|a}\frac{a}{k}x^{a-k}(1-2x^k+x^{2k}).\]

\subsection{Asymptotics} \label{subsec-asymptotics}

For a fixed $g\geqslant 2$, we consider the refined invariant as a function
$$\overline{BG}_g\colon\BBB\to\QQ[\![x]\!],$$
where the codomain $\QQ[\![x]\!]$ is the ring of formal series in $x$, endowed with the ultrametric distance: $d(f(x),g(x))= e^{-M}$ where
$$f(x)-g(x)=\alpha_Mx^M+o(x^M),\ \alpha_M\neq 0.$$
Thus, it makes sense to speak about asymptotic development of $\overline{BG}_g$ when $\det(B)$ is big enough. The benefit of considering formal series in $x$ rather than just polynomials is that the codomain is a complete space. Of course, the values of $\overline{BG}_g$ are actually polynomials.

\begin{prop}
    There is a polynomial function $AR_g\colon\NN\to\QQ[\![x]\!]$ (where ``AR'' stands for Asymptotic Refined) of degree at most $g-2$, equivalently an element of $\QQ[\![x]\!][n]$ such that inside $\QQ[\![x]\!]$, such that
    $$\overline{BG}_{g,B}(x) = AR_g(\det(B),x)+o(1).$$
    The $o(1)$ is the Landeau notation using the ultrametric topology on $\QQ[\![x]\!]$. Concretely, this means that for any $i_0>0$, we have equality of the $x^i$-coefficients for $0\leqslant i\leqslant i_0$, provided that $\det(B)$ is big enough.    
\end{prop}

\begin{proof}
    This is just a reformulation of Theorem \ref{theo:polynomiality}, taking $AR_g(n,x)=\sum_{i=0}^\infty Q_{g,i}(n)x^i$.
\end{proof}

\begin{rem}
One other way to formulate the asymptotic development is that for every fixed $i,g$ and $n$ chosen big enough with respect to $i$ and $g$, we have the congruence
$$\overline{BG}_{g,B}(x)\equiv AR_g(n,x) \mod x^{i+1}.$$
Thus, to compute $AR_g$, it suffices to compute $\overline{BG}_{g,B}(x)$ modulo $x^{i+1}$ for $n=\det(B)$ big enough and any $i$. We refer to the examples section for concrete computations.
\end{rem}

We also have an asymptotic refined invariant $AR_g^\star$, associated to $\bar{BG}^\star_{g,B}(q)$ counting curves in a fixed linear system. The relation to $AR_g$ is just a multiplication by $\frac{n}{g(g-1)}$.

\subsection{Statement of the result} \label{subsec-statement}

We now state the main regularity result concerning the asymptotic refined invariant. For aesthetic reasons, we state the result for $AR_g^\star$.

\begin{theo}\label{theorem-abelian-codeg-series}
    For fixed genus $g$, there exists quasi-modular forms $f_1,\dots,f_{g-3}$ vanishing at $0$ such that
    $$ AR_g^\star(n,x)= \bino{n}{g-1}+\sum_{k=1}^{g-3} f_k(x)n^k \in \QQ[\![x]\!][n].$$
\end{theo}

Before proving the theorem, as an illustration, we provide the first values of the asymptotic refined invariant. Explicit computation can be found in Section \ref{subsec-firstvalues}.

\begin{prop}\label{prop-firstvalues}
    We have the following expressions:
    \begin{enumerate}[label=(\roman*)]
        \item $AR_2^\star(n,x) = n$;
        \item $AR_3^\star(n,x) = \displaystyle\binom{n}{2}$;
        \item $AR_4^\star(n,x) = \displaystyle\binom{n}{3}-2E_2(x)n$, where $E_2(x)=\sum_{a=1}^\infty \sigma_1(a)x^a$.
    \end{enumerate}
\end{prop}

In particular, the first two asymptotic refined invariants have no dependence in $x$, meaning all $Q_{g,i}$ are $0$ except $Q_{g,0}$. The case $g=4$ is the first value for which we see a quasi-modular form appear, the sum in Theorem \ref{theorem-abelian-codeg-series} being $0$ for $g=2,3$.

\section{Proof of Theorem \ref{theorem-abelian-codeg-series} on regularity properties of asymptotic invariants}
\label{sec:proofs}

In this section we study the asymptotic invariant $AR_g^\star (x)$ to prove Theorem \ref{theorem-abelian-codeg-series}. The proof being technical, we first deal with the particular cases of Proposition \ref{prop-firstvalues} to illustrate the technique. We hope the details given below in Section \ref{subsec-firstvalues} explain how to tackle the general case in Section \ref{subsec-mainproof}. We will heavily use the formula
\begin{equation}\label{formule-BGbar}
    \overline{BG}_{g,n}^\star(x) = \dfrac{n}{g-1}\sum_{a_1+\cdots+a_{g-1}=n}\bar{P}_{a_1}\cdots\bar{P}_{a_{g-1}}, \quad \text{where }\bar{P}_a(x) = \sum_{k|a}\frac{a}{k}x^{a-k}(1-2x^k+x^{2k}).
\end{equation}

\subsection{Computation of the first values} \label{subsec-firstvalues}

\begin{proof}[Proof of Proposition \ref{prop-firstvalues}(i)]
We start with the easiest case $g=2$. By Formula (\ref{formule-BGbar}) we have $\overline{BG}^\star_{2,n}(x) = n\bar{P}_n(x)$. Moreover, if $n>2i$ and $k$ is a strict divisor of $n$, we have $n-k\geqslant\frac{n}{2}>i$. In particular, only the $k=n$ terms contribute to the first $i$ coefficients of $P_n(x)$, so that $\overline P_n(x)\equiv 1\mod x^{i+1}$. Hence, all the coefficients are asymptotically $0$ except the constant term, and we get
    $$AR_2^\star (n,x)=n.$$
\end{proof}
    
\begin{proof}[Proof of Proposition \ref{prop-firstvalues}(ii)]
We continue with $g=3$. By Formula (\ref{formule-BGbar}) the invariant is given by
    $$ \bar{BG}_{3,n}^\star(x) = \frac{n}{2}\sum_{a_1+a_2=n}\bar P_{a_1} \bar P_{a_2}.$$
If $n>4i$, then either $a_1 > 2i$ or $a_2 > 2i$ and we have $\bar P_{a_j}\equiv 1\mod x^{i+1}$ for $j=1$ or $2$. Thus, we split the sum as follows:
 \begin{align*}
       \bar{BG}_{3,n}^\star(x) 
        & \equiv \frac{n}{2} \left( \sum_{\substack{a_1+a_2=n \\ a_1\leqslant 2i}}\bar P_{a_1} + \sum_{\substack{a_1+a_2=n \\ a_1,a_2>2i}}1 + \sum_{\substack{a_1+a_2=n \\ a_2\leqslant 2i}}\bar P_{a_2} \right) \mod x^{i+1}.
    \end{align*}
Next, we would like to remove the condition $a_1,a_2>2i$ from the middle sum. To do so, we write each of the $\bar{P}_a$ as $(\bar{P}_a-1)+1$ and incorporate the $1$ terms to the middle sum:
    \[   \bar{BG}_{3,n}^\star(x) \equiv \frac{n}{2} \left( \sum_{\substack{a_1+a_2=n \\ a_1\leqslant 2i}}(P_{a_1}-1) + \sum_{a_1+a_2=n}1 + \sum_{\substack{a_1+a_2=n \\ a_2\leqslant 2i}}(P_{a_2}-1) \right)  \mod x^{i+1}. \]
Using again that $\bar P_{a_j}\equiv 1\mod x^{i+1}$ for $a_j>2i$, we are able to add the missing terms to the first and third sums, ultimately yielding
\begin{align*}
           \bar{BG}_{3,n}^\star(x) &\equiv \frac{n}{2} \left( \sum_{a_1=1}^{+\infty} (P_{a_1}-1) + n-1 + \sum_{a_2=1}^{+\infty} (P_{a_2}-1) \right)  \mod x^{i+1}  \\
        &\equiv \frac{n}{2} \left( n-1 + 2\sum_{a=1}^{+\infty} (P_a-1) \right) \mod x^{i+1}.
    \end{align*}
    Since the right-hand side does snot depend on $i$ and the equality is valid modulo any $x^{i+1}$, the series is the sought $AR^\star_3(x)$. Notice that $d(P_a-1,0)$ goes to $0$ when $a$ goes to $\infty$, so that $\sum_{a=1}^\infty (P_a-1)$ makes sense in the ring of formal series $\QQ[\![x]\!]$. An elementary computation, which is part of Lemma \ref{lemma-Gm-quasi-modular} proven hereby after, one has $\sum_{a=1}^\infty (P_a-1)=0$, and we get the expected result:
    \[ AR_3^\star (n,x)= \frac{n(n-1)}{2} = \binom{n}{2} . \]
\end{proof}

\begin{proof}[Proof of Proposition \ref{prop-firstvalues}(iii)]
Last, we compute the asymptotic invariant for $g=4$. By Formula (\ref{formule-BGbar}) the invariant is given by
    \[ \bar{BG}_{4,n}^\star(x) = \frac{n}{3}\sum_{a_1+a_2+a_3=n}\bar P_{a_1}\bar P_{a_2}\bar P_{a_3}.\]
We apply the same method: the sum is over the integral points of the triangle
$$\Delta_n=\{(a_1,a_2,a_3)\ |\ a_1+a_2+a_3=n \text{ and } a_j\geq1\}.$$
We split each $\bar{P}_a$ as $(\bar{P}_a-1)+1$, yielding
    $$\prod_{i=1}^{g-1}\bar{P}_{a_i} = \sum_{I\subset [\![1;g-1]\!]}\prod_{i\in I}(\bar{P}_{a_i}-1).$$
    In the $g=4$ case, we get
    \begin{equation}\label{genus4-decompo}
        \bar{P}_{a_1}\bar{P}_{a_2}\bar{P}_{a_3} = \begin{array}{l}
         (\bar{P}_{a_1}-1)(\bar{P}_{a_2}-1)(\bar{P}_{a_3}-1) \\
         + (\bar{P}_{a_1}-1)(\bar{P}_{a_2}-1)+(\bar{P}_{a_2}-1)(\bar{P}_{a_3}-1)+(\bar{P}_{a_1}-1)(\bar{P}_{a_3}-1) \\
         + (\bar{P}_{a_1}-1)+(\bar{P}_{a_2}-1)+(\bar{P}_{a_3}-1) \\
         + 1.
        \end{array}
    \end{equation}
    We thus get a sum of eight functions over $\Delta_n$, some of them playing a symmetric role up to permuting the indices. In each case, since $\bar{P}_a-1\equiv 0 \mod x^{i+1}$ if $a>2i$, only the $(a_1,a_2,a_3)$ where the $a_j$ appearing in the expression are smaller than $2i$ contribute. hence, the idea is to split the triangle according to how close the points are to the boundary and corners of $\Delta_n$, because this forces the vanishing of the function. Assume $i\geqslant 0$ and $n>6i$.
    \begin{itemize}
        \item First, we have
        $$\sum_{\substack{a_1+a_2+a_3=n \\ a_1,a_2,a_3\leqslant 2i}}(\bar{P}_{a_1}-1)(\bar{P}_{a_2}-1)(\bar{P}_{a_3}-1) = 0,$$
        since due to $n>6i$, the summation set is actually empty.
        \item Then, for each of the three functions appearing in the third row of (\ref{genus4-decompo}), we have
        \begin{align*}
          \sum_{\substack{a_1+a_2+a_3=n \\ a_1\leqslant 2i}}(\bar{P}_{a_1}-1) = &   \sum_{a_1=1}^{2i}(\bar{P}_{a_1}-1)\left(\sum_{a_2+a_3=n-a_1}1\right) \\
          = & \sum_{a_1=1}^{2i}(n-1-a_1)(\bar{P}_{a_1}-1) \\
          \equiv & (n-1)\sum_{a=1}^\infty (\bar{P}_a-1) - \sum_{a=1}^\infty a(\bar{P}_a-1) \mod x^{i+1}.
        \end{align*}
        To get to the last row, we just added the terms for $a_1>2i$ since they do not change the value of the sum modulo $x^{i+1}$. Elementary computations proven in Lemma \ref{lemma-Gm-quasi-modular} yield that $\sum_{a=1}^\infty a(\bar{P}_a-1)=2E_2(x)$, where $E_2(x)=\sum_1^\infty \sigma_1(a)x^a$ is the first Eisenstein series, normalized to have $0$ first coefficient. We already know that $\sum_1^\infty(\bar{P}_a-1)=0$.
        \item For the function involving two $a_j$ out of the three (second row of (\ref{genus4-decompo}), we proceed similarly:
        \begin{align*}
        \sum_{\substack{a_1+a_2+a_3=n \\ a_1,a_2\leqslant 2i}}(\bar{P}_{a_1}-1)(\bar{P}_{a_2}-1) = & \sum_{a_1,a_2=1}^{2i}(\bar{P}_{a_1}-1)(\bar{P}_{a_2}-1) \\
        = & \left(\sum_{a=1}^\infty (\bar{P}_a-1)\right)^2 \equiv 0 \mod x^{i+1},
        \end{align*}
        using again that $\sum_{a=1}^\infty (\bar{P}_a-1)\equiv 0\mod x^{i+1}$.
        \item Last, we have that
        $$\sum_{a_1+a_2+a_3=n}1 = \binom{n-1}{2}.$$
    \end{itemize}
    In total, we get
    $$\bar{BG}^\star_{4,n}(x) \equiv\frac{n}{3}\left( 0 +3\cdot 0+3(0-2E_2(x))+\binom{n-1}{2} \right) \mod x^{i+1},$$
    yielding the expected result:
    \[ AR_4^\star(n,x) = \frac{n}{3} \left( \binom{n-1}{2} - 6E_2 \right) = \binom{n}{3} - 2nE_2(x). \]
\end{proof}

\subsection{General result}

\subsubsection{An auxiliary family of quasi-modular forms}
Before getting to the proof of Theorem \ref{theorem-abelian-codeg-series}, we introduce a family of generating series $(G_m(x))_m$ that appears throughout the computation, and prove they are quasi-modular forms. The generating series $G_0$ and $G_1$ are the one appearing in the $g=3$ and $g=4$ cases of Proposition \ref{prop-firstvalues}. We consider the function
$$G_m(x)=\sum_{a=1}^{+\infty} a^m(\bar P_a(x)-1)$$
and denote by $E_{2j}$ the Eisenstein series with $0$ constant term, \ie 
\[ E_{2j}(x)=\sum_{a=1}^{+\infty} \sigma_{2j-1}(a)x^a \]
with $\sigma_k(a)=\sum_{d|a}d^k$. 
Let $D$ be the differential operator $D=x\frac{\mathrm{d}}{\mathrm{d}x}$. We recall that if $F(x)$ is a quasi-modular form, then so is $DF(x)$.

\begin{lem}\label{lemma-Gm-quasi-modular}
    The function $G_m(x)$ is a quasi-modular form that vanishes at $0$. Moreover, we have the explicit formula
    $$G_m(x) = 2\sum_{2\leqslant 2j\leqslant m+1}\bino{m+1}{2j} D^{m+1-2j}E_{2j}(x) . $$
\end{lem}

\begin{expl}
    For small $m$, the explicit expression yields the following identities:
    \begin{itemize}[label=$\circ$]
        \item $G_0(x)=0$,
        \item $G_1(x)=2E_2(x)$,
        \item $G_2(x)=6DE_2(x)$,
        \item $G_3(x)=12D^2E_2(x)+2E_4(x)$,
        \item $G_4(x)=20D^3E_2(x)+10DE_4(x)$,
        \item $G_5(x)=30D^4E_2(x)+30D^2E_4(x)+2E_6(x)$.
    \end{itemize}
\end{expl}

\begin{proof}
    As there is no constant term, the vanishing part is obvious. To show the quasi-modularity and the formula, we proceed as follows. We start with the expression of $\bar{P}_a$, and write $1=\sum_{kl=a}l\delta_{l,1}$, so that we have:
    \begin{align*}
        \bar P_a-1 & = \sum_{kl=a} lx^{k(l-1)}(1-2x^k+x^{2k}) - \sum_{kl=1}l\delta_{l,1} \\
        = & \sum_{kl=a}l \left(x^{k(l-1)}-2x^{kl}+x^{k(l+1)}-\delta_{l,1}\right).
    \end{align*}
    Hence we get
    \begin{align*}
        G_m(x) &= \sum_{k,l=1}^\infty (kl)^m l\left( x^{k(l-1)}-\delta_{l,1}-2x^{kl}+x^{k(l+1)}\right) \\
        &= \sum_{k,l=1}^\infty k^ml^{m+1}(x^{k(l-1)}-\delta_{l,1}) -2 \sum_{k,l=1}^\infty  (kl)^mlx^{kl} + \sum_{k,l=1}^\infty k^ml^{m+1} x^{k(l+1)}.
    \end{align*}
    The summand of the first sum vanishes if $l=1$, thus we start at $l=2$ and set $l'=l-1$. For the last sum, we can add the summand $l=0$ since $l^{m+1}$ is $0$. Setting in this last sum $l'=l+1$, we get in total:
    \[ G_m(x) = \sum_{k,l'=1}^\infty k^m(l'+1)^{m+1}x^{kl'} -2\sum_{k,l=1}^\infty (kl)^mlx^{kl}+\sum_{k,l'=1}^\infty k^m(l'-1)^{m+1}x^{kl'}.
    \]
    We drop the prime on $l'$, use the binomial formula in the first and third sums and gather them:
    \begin{align*}
        G_m(x) &= \sum_{j=0}^{m+1} \bino{m+1}{j} (1+(-1)^j)\sum_{k,l=1}^\infty k^m l^{m+1-j}x^{kl} -2\sum_{k,l=1}^\infty (kl)^mlx^{kl}\\
        &= 2\sum_{0\leq2j\leq m+1} \bino{m+1}{2j} \sum_{k,l=1}^\infty k^m l^{m+1-2j}x^{kl} -2\sum_{k,l=1}^\infty (kl)^mlx^{kl}\\
        &= 2 \sum_{2\leq2j\leq m+1} \bino{m+1}{2j} \sum_{k,l=1}^\infty (kl)^{m+1-2j} k^{2j-1} x^{kl} .
    \end{align*}
    To get to the second row, we noticed that $1+(-1)^j$ is non-zero only when $j$ is even, and we thus only sum over even numbers. To get to the last row, we use that the last sum cancels with the $2j=0$ term in the first sum. To conclude we use that:   
    $$\sum_{k,l=1}^\infty (kl)^{m+1-2j}k^{2j-1} x^{kl} =  \sum_{a=1}^\infty a^{m+1-2j}\sigma_{2j-1}(a)x^a = D^{m+1-2j}E_{2j}(x).$$
    The ring of quasi-modular forms is stable by $D$, so that $D^{m+1-2j}E_{2j}(x)$ are quasi-modular forms. Thus, as $G_m$ is a polynomial in derivatives of quasi-modular forms, it is also a quasi-modular form, yielding the result.    
\end{proof}

\subsubsection{Proof of the main result.}\label{subsec-mainproof} 

We now get to prove Theorem \ref{theorem-abelian-codeg-series}. The idea is to split the sum over partitions of $n$ according to the distance to a face of the simplex $\{a_1+\cdots+a_{g-1}=n\}$, adding the missing terms so that we get sums over integral points of simplices of size $n$. These expression make the $G_m(x)$ appear naturally.

\begin{lem}\label{lem-sommeproduitPa}
    One has 
    \begin{align*}
        \sum_{a_1+\cdots+a_{g-1}=n}\bar P_{a_1}\cdots \bar P_{a_{g-1}} 
        &\equiv \sum_{s=0}^{g-2} \binom{g-1}{s}\sum_{\substack{a_1+\cdots+a_{g-1}=n \\ a_1,\dots,a_s\leqslant 2i }}(\bar P_{a_1}-1)\cdots (\bar P_{a_s}-1) \mod x^{i+1} . 
    \end{align*}
\end{lem}

\begin{proof}
    The sum runs over the set of integer points of the simplex $\{a_1+\cdots a_{g-1}=n\}$. We first write each $\bar{P}_{a_j}$ as $\bar{P}_{a_j}-1+1$ and expand the inner product:
    $$\prod_{j=1}^{g-1}\bar{P}_{a_j} = \sum_{I\subset[\![1;g-1]\!]} \prod_{j\in I}(\bar{P}_{a_j}-1).$$
    We then sum over the integral points of the simplex. Since $\bar{P}_{a_j}-1$ vanishes modulo $x^{i+1}$ if $a_j>2i$, we are left with
    \begin{align*}
        \sum_{a_1+\cdots+a_{g-1}=n}\bar P_{a_1}\cdots \bar P_{a_{g-1}} \equiv & \sum_{I\subset[\![1;g-1]\!]} \sum_{\substack{a_1+\cdots+a_{g-1}=n \\ a_j\leqslant 2i \text{ if }j\in I}}\prod_{j\in I}(\bar{P}_{a_j}-1) \mod x^{i+1} \\
        \equiv & \sum_{s=0}^{g-1} \binom{g-1}{s} \sum_{\substack{a_1+\cdots+a_{g-1}=n \\ a_1,\dots,a_s\leqslant 2i}}\prod_{j=1}^s(\bar{P}_{a_j}-1) \mod x^{i+1},
    \end{align*}
    where we noticed that the inner sum only depends on the cardinality $s$ of the subset $I$, and that there are $\binom{g-1}{s}$ subsets of cardinality $s$. Finally, the $s=g-1$ term is $0$ since for $n>2(g-1)i$, there is no choice of $(a_j)$ all of them smaller than $2i$ and sum equal to $n$.
\end{proof}

\begin{proof}[Proof of Theorem \ref{theorem-abelian-codeg-series}]
By Formula (\ref{formule-BGbar}) we have to compute
\[ \dfrac{n}{g-1}\sum_{a_1+\cdots+a_{g-1}=n}\bar{P}_{a_1}\cdots\bar{P}_{a_{g-1}} \mod x^{i+1},\]
and by Lemma \ref{lem-sommeproduitPa} this amounts to determine
\[ \frac{n}{g-1}\sum_{s=0}^{g-2} \bino{g-1}{s}\sum_{\substack{a_1+\cdots+a_{g-1}=n \\ a_1,\dots,a_s\leqslant 2i }}(\bar P_{a_1}-1)\cdots (\bar P_{a_s}-1) \mod x^{i+1}. \]
We will deal with this expression, forgetting the factor $\frac{n}{g-1}$ for the moment.

Since $|\{a_{s+1}+\cdots+a_{g-1}=n-a_1-\cdots-a_s\}|=\binom{n-1-\sum_1^s a_j}{g-2-s}$ we get
    \begin{align*}
        \sum_{s=0}^{g-2} \bino{g-1}{s}\sum_{ a_1,\dots,a_s\leqslant 2i} \binom{n-1-\sum_1^s a_j}{g-2-s} (\bar P_{a_1}-1)\cdots (\bar P_{a_s}-1) \mod x^{i+1} . 
    \end{align*}
    We add the missing terms for $a_j>2i$. The latter do not change the value modulo $x^{i+1}$ since $\bar P_{a_j}-1\equiv 0\mod x^{i+1}$ if $a_j>2i$. We get
    \begin{align*}
         \sum_{s=0}^{g-2} \bino{g-1}{s}\sum_{ a_1,\dots,a_s=1}^\infty \binom{n-1-\sum_1^s a_j}{g-2-s} (\bar P_{a_1}-1)\cdots (\bar P_{a_s}-1) \mod x^{i+1} . 
    \end{align*}
    
    We put aside the $s=0$ term, equal to $\sbino{n-1}{g-2}$. The multiplication by $\frac{n}{g-1}$ yields $\binom{n}{g-1}$ as announced.
    
    If $s>0$, we see the binomial coefficient as the corresponding Hilbert polynomial, defined by $H_m(n)=\binom{n}{m}$ and which is of degree $m$. Then $\sbino{n-1-\sum_1^s a_j}{g-2-s} = H_{g-2-s}(n-1-\sum_1^s a_j)$ is a polynomial in $a_1,\dots,a_s$ with coefficients being polynomials in $n$. Given a monomial $a_1^{m_1}\cdots a_s^{m_s}$, the following sum factors in a product:
    \[ \sum_{a_1,\dots,a_s=1}^\infty \prod_{j=1}^s a_j^{m_j}(\bar P_{a_j}-1)=\prod_{j=1}^s G_{m_j}(x), \]
    and thus is a quasi-modular form since the $G_m$ are by Lemma \ref{lemma-Gm-quasi-modular}. This concludes for the quasi-modularity.
    
    The maximal degree in $n$ is achieved for $s=1$, with
    $$\sum_{a=1}^\infty H_{g-3}(n-1-a) (\bar P_a-1),$$
    which is a polynomial in $a,n$ of degree $g-3$. However, as $G_0=0$ by Lemma \ref{lemma-Gm-quasi-modular}, the summation over $a$ cancels the terms where the exponent of $a$ is $0$, and in particular the coefficient of $n^{g-3}$. Therefore, the degree in $n$ after summation over $a$ is $g-4$. The multiplication by $\frac{n}{g-1}$ yields the announced degree $g-3$.
\end{proof}

\begin{rem}
    More generally, to give a non-zero contribution, a monomial in $a_1,\dots,a_s$ needs to be of degree at least $1$ in each $a_j$. In particular, this forces $g-2-s\geqslant s$, or in other terms $s\leqslant\frac{g-2}{2}$.
\end{rem}

We recover the values computed in Proposition \ref{prop-firstvalues}, but also get new values.

\begin{expl}
Now, we take $g=5$. Due to the condition $2s\leqslant g-2=3$, the index $s$ can take the values $s=0,1$. The value $s=0$ yields the constant term. For $s=1$, we have
    $$H_2(n-1-a)=\frac{(n-a-1)(n-a-2)}{2}=\frac{1}{2}a^2+\frac{3-2n}{2}a+\bino{n}{2},$$
    so that
    \begin{align*}
        AR_5^\star(n,x) &= \bino{n}{4} + \frac{n}{4}\cdot 4\left( \frac{1}{2}G_2(x)+\frac{3-2n}{2}G_1(x) \right) \\
        &= \bino{n}{4} + \left(\frac{3}{2}G_1(x)+\frac{1}{2}G_2(x)\right)n-G_1(x)n^2 . 
    \end{align*}
\end{expl}

\begin{expl}
We now take $g=6$. In that case, we may have $s=0,1$ or $2$ since $2s\leqslant g-2=4$. The value $s=0$ yields the constant term. For $s=1$, we have
    \begin{align*}
        H_3(n-1-a) &= -\frac{(a-n+1)(a-n+2)(a-n+3)}{6} \\
        &= -\frac{1}{6}a^3+\frac{3n-6}{6}a^2-\frac{3n^2-12n+11}{6}a+\bino{n}{3}, 
    \end{align*}
    and for $s=2$,
    $$H_2(n-1-a_1-a_2)=a_1a_2+\cdots.$$
    Only the term $a_1a_2$ is of interest since the other monomials do not contribute: only the monomials with degree at least $1$ in each variable contribute. Thus, we have
    \begin{align*}
        AR_6^\star(n,x) &= \bino{n}{5} + \frac{n}{5}\left[ 5\left( -\frac{1}{6}G_3+\frac{3n-6}{6}G_2-\frac{3n^2-12n+11}{6}G_1\right) + 10G_1^2 \right] \\
        &= \bino{n}{5} + \left( 2G_1^2-G_2-\frac{1}{6}G_3-\frac{11}{6}G_1\right)n+\left(\frac{1}{2}G_2-2G_1\right)n^2-\frac{1}{2}G_1n^3.
    \end{align*}
\end{expl}

\section{First coefficients of the series in fixed codegree} \label{sec-fixedcodeg}

We defined the asymptotic invariant $AR_g^\star$, which is a polynomial in $n$ with coefficients in $\QQ[\![x]\!]$, and depending on $g$. We now adopt an orthogonal point of view studying the generating series in the parameter $g$:
$$AR^\star=\sum_{g=2}^\infty AR_g^\star u^g \in\QQ[n][\![x,u]\!],$$
which is now a formal series in parameters $x$ and $u$ with coefficients being polynomials in $n$. We first care about the generating series with the genus parameter and fixed codegree, meaning we consider $AR^\star$ modulo $x^{i+1}$ for some $i$. 
We will again use extensively Formula~(\ref{formule-BGbar}). If $Q$ is a polynomial, we denote by $\coef{Q}_i$ its degree $i$ coefficient. Recall that the binomial coefficient $\binom{n}{k}$ is $0$ if $k<0$.

\begin{prop}
    The first coefficients of $AR_g^\star$ are given by the following expressions:
    \begin{enumerate}[label=(\roman*)]
        \item $\coef{AR_g^\star}_0 = \displaystyle\binom{n}{g-1}$,
        \item $\coef{AR_g^\star}_1 = \displaystyle-2n\binom{n-3}{g-4}$, 
        \item $\coef{AR_g^\star}_2 = \displaystyle n \left( \binom{n-2}{g-3} - 6\binom{n-3}{g-3} + 3\binom{n-4}{g-3} - 8\binom{n-5}{g-6} - 2(n-5)\binom{n-6}{g-7} \right)$,
    \end{enumerate}
\end{prop}

\begin{proof}
    \begin{enumerate}[label=(\roman*)]
    \item The $0$-coefficients are the constant terms and have thus already been computed in Theorem \ref{theorem-abelian-codeg-series}, using that all $\bar{P}_a$ have constant term $1$. For convenience of the reader, we recall the proof here.
    Using $\coef{\bar P_a}_0=1$ and Formula (\ref{formule-BGbar}) one has 
    \[ \coef{BG^\star _{g,n}(x)}_0 =  \frac{n}{g-1}\sum_{a_1+\cdots a_{g-1}=n}1 = \frac{n}{g-1}\bino{n-1}{g-2}=\bino{n}{g-1}. \]

    \item Using Formula (\ref{formule-BGbar}), since the constant terms of $\bar{P}_a$ are $1$, we have
    \begin{align*}
         \coef{BG^\star _{g,n}(x)}_1 &= \frac{n}{g-1}\sum_{a_1+\cdots+a_{g-1}=n} \coef{\bar{P}_{a_1}\cdots\bar{P}_{a_{g-1}}}_1 \\
         &= \frac{n}{g-1}\sum_{a_1+\cdots+a_{g-1}=n}\left( \coef{\bar{P}_{a_1}}_1+\dots+\coef{\bar{P}_{a_{g-1}}}_1 \right)\\
         &= n\sum_{a_1+\cdots+a_{g-1}=n} \coef{\bar{P}_{a_1}}_1.
    \end{align*}
    Note that the degree $1$ coefficient of $\bar{P}_a$ is non-zero only for $a=1,2$ for which one has
    \begin{align*}
        \bar{P}_1(x)= & 1-2x+x^2, \\
        \bar{P}_2(x)= & 1+2x-6x^2+2x^3+x^4.
    \end{align*}
    Thus, in the sum, the value of $a_1$ can only be 1 or 2 and we get
    \begin{align*}
        \coef{BG^\star_{g,n}(x)}_1 &= n\left( \sum_{a_2+\cdots+a_{g-1}=n-1}\coef{\bar P_1}_1 + \sum_{a_2+\cdots+a_{g-1}=n-2}\coef{\bar P_2}_1 \right) \\
        &= 2n\left( \bino{n-3}{g-3}-\bino{n-2}{g-3} \right) = -2n\bino{n-3}{g-4}.
    \end{align*}

    \item  Using Formula (\ref{formule-BGbar}), since the constant terms of $\bar{P}_a$ are $1$, we have
    \begin{align*}
         \coef{BG^\star _{g,n}(x)}_2 &= \frac{n}{g-1}\sum_{a_1+\cdots+a_{g-1}=n} \left(\sum_{k=1}^{g-1} \coef{\bar{P}_{a_i}}_2 + \sum_{1\leq i<j\leq g-1} \coef{\bar P_{a_i}}_1 \coef{\bar P_{a_j}}_1 \right) \\
         &=n\sum_{a_1+\cdots+a_{g-1}=n} \coef{\bar{P}_{a_1}}_2 + \frac{n(g-2)}{2} \sum_{a_1+\cdots+a_{g-1}=n} \coef{\bar P_{a_1}}_1 \coef{\bar P_{a_2}}_1 .
    \end{align*}
    We now look for values of $(a_i)$ where the summand is non-zero.
        \begin{itemize}
            \item For the first sum, the only $\bar{P}_a$ with non-zero degree $2$ terms are $\bar{P}_1$, $\bar{P}_2$ and
    $$\bar{P}_3=1+3x^2-8x^3+3x^4+x^6.$$
    Hence one has
    \begin{align*}
        \sum_{a_1+\cdots+a_{g-1}=n} \coef{\bar{P}_{a_1}}_2 
        &= \sum_{a_2+\dots+a_{g-1}=n-1} \coef{\bar P_1}_2 + \sum_{a_2+\dots+a_{g-1}=n-2} \coef{\bar P_2}_2 + \sum_{a_2+\dots+a_{g-1}=n-3} \coef{\bar P_3}_2\\
        &= \binom{n-2}{g-3} -6\binom{n-3}{g-3} +3\binom{n-4}{g-3} .
    \end{align*}
            
            \item In the second sum, $a_1$ and $a_2$ can only be 1 or 2, so that this term is
    \begin{align*}
        &\ \sum_{a_1+\cdots+a_{g-1}=n} \coef{\bar P_{a_1}}_1 \coef{\bar P_{a_2}}_1 \\
        =&\  \sum_{a_3+\cdots+a_{g-1}=n-2} \coef{\bar P_{1}}_1^2 + 2 \sum_{a_3+\cdots+a_{g-1}=n-3} \coef{\bar P_{1}}_1 \coef{\bar P_{2}}_1 +  \sum_{a_3+\cdots+a_{g-1}=n-4} \coef{\bar P_{2}}_1^2  \\
        =&\  4 \left(  \binom{n-3}{g-4} -2\binom{n-4}{g-4} + \binom{n-5}{g-4} \right)  = 4 \binom{n-5}{g-6}.
        \end{align*}
        \end{itemize}
    
    Putting all together, taking into account the respective factors $n$ and $\frac{n(g-2)}{2}$ in front of the sums, we obtain
    \begin{align*}
     \coef{BG^\star _{g,n}(x)}_2 
     &= n\left( \binom{n-2}{g-3} -6\binom{n-3}{g-3} +3\binom{n-4}{g-3}  -2(g-2) \binom{n-5}{g-6} \right) \\
     &= n \left( \binom{n-2}{g-3} - 6\binom{n-3}{g-3} + 3\binom{n-4}{g-3} - 8\binom{n-5}{g-6} - 2(g-6)\binom{n-5}{g-6} \right) \\
     &= n \left( \binom{n-2}{g-3} - 6\binom{n-3}{g-3} + 3\binom{n-4}{g-3} - 8\binom{n-5}{g-6} - 2(n-5)\binom{n-6}{g-7} \right) . 
     \end{align*}
    \end{enumerate}
\end{proof}

\begin{coro}
    The generating series of the coefficients are:
    \begin{enumerate}[label=(\roman*)]
        \item for constant term, $u(1+u)^n-u$,
        \item for first coefficient, $-2nu^4(1+u)^{n-3}$,
        \item for second coefficient, $u(1+u)^n \left[  -2n^2 \frac{u^6}{(1+u)^6} + n\frac{3u^6 -10 u^5 - 9 u^4 - 8 u^3 - 2u^2}{(1+u)^6} \right]$.
    \end{enumerate}
\end{coro}

\begin{proof}    
    \begin{enumerate}[label=(\roman*)]
    \item Using the binomial formula, the generating series is the announced result:
    $$\sum_{g=2}^\infty \binom{n}{g-1}u^g = u(1+u)^n-u.$$

    \item The generating series is
    $$-2n\sum_{g=2}^\infty \binom{n-3}{g-4}u^g = -2nu^4(1+u)^{n-3}.$$

    \item Finally, computing the generating series of second coefficients for $g\geq2$ yields:
     \begin{align*}
         &\ n\left[ u^3(1+u)^{n-2} - 6u^3(1+u)^{n-3} + 3u^3(1+u)^{n-4} - 8u^6(1+u)^{n-5} - 2(n-5)u^7(1+u)^{n-6}  \right] \\
        =&\ nu^3(1+u)^{n-6} \left[ (1+u)^4 - 6(1+u)^3 + 3(1+u)^2 - 8u^3(1+u) - 2(n-5)u^4 \right] \\
        =&\ nu^3(1+u)^{n-6} \left[-2 n u^4 +3u^4 -10 u^3 - 9 u^2 - 8 u - 2 \right] \\
        =&\ u(1+u)^n \left[  -2n^2 \frac{u^6}{(1+u)^6} + n\frac{3u^6 -10 u^5 - 9 u^4 - 8 u^3 - 2u^2}{(1+u)^6} \right]. 
     \end{align*}
    \end{enumerate}
\end{proof}

\begin{rem}
In particular, we have that
$$AR^\star \equiv u(1+u)^n\left(1 - 2n\frac{u^3}{(1+u)^3}x \right) -u \mod x^2.$$
Taking into account the next coefficient, the generating series does not seem to have a compact form.
\end{rem}

\bibliographystyle{alpha}
\bibliography{biblio}

\end{document}

%% file: preamble.tex
\usepackage[utf8]{inputenc}
\usepackage[T1]{fontenc}
\usepackage[english]{babel}
\usepackage{amsmath}
\usepackage{amsthm}
\usepackage{amssymb}
\usepackage{mathrsfs}
\usepackage{dsfont}
\usepackage{graphicx} 
\usepackage[left=3cm,right=3cm,top=3cm,bottom=3cm]{geometry}
\usepackage{hyperref}

\usepackage{enumitem}
\usepackage{tabularray}

\usepackage{eucal} 

\usepackage{mlmodern}

\newcommand{\PP}{\mathbb{P}}
\newcommand{\NN}{\mathbb{N}}
\newcommand{\ZZ}{\mathbb{Z}}
\newcommand{\RR}{\mathbb{R}}
\newcommand{\CC}{\mathbb{C}}

\newcommand{\QQ}{\mathbb{Q}}

\newcommand{\TT}{\mathbb{T}}

\newcommand{\BBB}{\mathscr{B}}

\renewcommand{\P}{\mathcal{P}}

\newcommand{\M}{\mathcal{M}}

\newcommand{\bra}[1]{\left\langle #1 \right\rangle}
\newcommand{\coef}[1]{\left[ #1 \right]}

\newcommand{\pt}{\mathrm{pt}}
\newcommand{\ev}{\mathrm{ev}}
\newcommand{\ft}{\mathrm{ft}}

\renewcommand{\geq}{\geqslant}
\renewcommand{\leq}{\leqslant}
\renewcommand{\bar}{\overline}

\newcommand{\bino}[2]{\begin{pmatrix} #1 \\ #2 \\ \end{pmatrix}}
\newcommand{\sbino}[2]{\left(\begin{smallmatrix} #1 \\ #2 \\ \end{smallmatrix}\right)}

\newtheorem{theo}{Theorem}[section]

\newtheorem{prop}[theo]{Proposition}
\newtheorem{coro}[theo]{Corollary}
\newtheorem{lem}[theo]{Lemma}

\newtheorem*{namedtheorem}{\theoremname}
 \newcommand{\theoremname}{Theorem}
 \newenvironment{named}[1]{\renewcommand{\theoremname}{#1}\begin{namedtheorem}}{\end{namedtheorem}}

\theoremstyle{definition}
\newtheorem{defi}[theo]{Definition}

\theoremstyle{remark}
\newtheorem{remark}[theo]{Remark}
\newenvironment{rem}[1]{
    \begin{remark}#1}{
    \xqed{\blacklozenge}\end{remark}
}

\theoremstyle{remark}
\newtheorem{example}[theo]{Example}
\newenvironment{expl}[1]{
    \begin{example}#1}{
    \xqed{\lozenge}\end{example}
}

\newcommand{\xqed}[1]{
    \leavevmode\unskip\penalty9999 \hbox{}\nobreak\hfill
    \quad\hbox{\ensuremath{#1}}}

\usepackage{todonotes}

\newcommand{\ie}{i.e. }

\newcommand{\nocontentsline}[3]{}
\newcommand{\tocless}[2]{\bgroup\let\addcontentsline=\nocontentsline#1{#2}\egroup}

\makeatletter
\def\l@subsection{\@tocline{2}{0pt}{2.5pc}{5pc}{}}
\makeatother

\makeatletter
\renewcommand{\l@section}{\@tocline{1}{0pt}{10pt}{1pc}{\bfseries}}
\makeatother